\documentclass[11pt]{amsart}
\usepackage{amsmath, amsfonts, amsthm, amssymb, multicol, mathtools, dsfont, verbatim, enumerate}
\usepackage{graphicx}
\usepackage{float, hyperref}
\usepackage{scalerel,stackengine, subcaption}
\usepackage[usenames,dvipsnames]{xcolor}
\usepackage{enumitem}
\usepackage{pgfplots}
\usepgfplotslibrary{fillbetween}
\pgfplotsset{width=10cm,compat=1.9}
\usepackage[square,sort,comma,numbers]{natbib}
\allowdisplaybreaks

\usepackage{fancyhdr}
 
\numberwithin{equation}{section}

\newtheorem{thm}{Theorem}[section]
\newtheorem{lma}[thm]{Lemma}
\newtheorem{cor}[thm]{Corollary}
\newtheorem{defn}[thm]{Definition}

\newtheorem{prop}[thm]{Proposition}

\newtheorem{ex}[thm]{Example}

\renewcommand{\epsilon}{\varepsilon}
\newcommand{\eps}{\varepsilon}

\newcommand{\dn}{\mathbb{D}^n}

\newcommand{\rn}{\mathbb{R}^n}
\newcommand{\radlim}{L_{rad}}

\newcommand{\geqs}{\gtrsim}
\newcommand{\bb}{\mathbb}
\newcommand{\cur}{\mathcal}
\newcommand{\steq}{\subseteq}
\newcommand{\LG}{L(\Gamma)}

\newcommand{\bdry}{S^{n-1}}

\renewcommand{\geq}{\geqslant}
\renewcommand{\leq}{\leqslant}
\renewcommand{\emptyset}{\varnothing}
\renewcommand{\epsilon}{\varepsilon}

\renewcommand{\geq}{\geqslant}
\renewcommand{\leq}{\leqslant}

\newcommand{\ubd}{\overline{\dim}_{\textup{B}}}

\newcommand{\lbd}{\underline{\dim}_{\textup{B}}}
\newcommand{\hd}{\dim_{\textup{H}}}

\newcommand{\bd}{\dim_{\textup{B}}}

\makeatletter
\def\@setauthors{%
  \begingroup
  \def\thanks{\protect\thanks@warning}%
  \trivlist
  \centering\footnotesize \@topsep30\p@\relax
  \advance\@topsep by -\baselineskip
  \item\relax
  \author@andify\authors
  \def\\{\protect\linebreak}

  \normalsize\lowercase{\authors}%
  
	\ifx\@empty\contribs
  \else
    ,\penalty-3 \space \@setcontribs
    \@closetoccontribs
  \fi
  \endtrivlist
  \endgroup
}
\def\@settitle{\begin{center}
\LARGE\lowercase{\@title}
  \end{center}%
}
\makeatother
\newcommand{\authoremail}[1]{\email{\href{mailto:#1}{\color{lightblue}{#1}}}}
\newcommand{\authoraddress}[1]{\address{\normalfont{#1}}}

\definecolor{lightblue}{HTML}{2B77A4}
\definecolor{darkred}{HTML}{9E0D0D}
\hypersetup{
	colorlinks=true,
	linkcolor=darkred,
	urlcolor=darkred,
	citecolor=lightblue
}
\title{Critical exponents and dimension for generalised limit sets}

\author{Tianyi Feng}
\authoraddress{Tianyi Feng, University of St Andrews, Scotland}
\authoremail{tf66@st-andrews.ac.uk}

\author{Jonathan M. Fraser}
\authoraddress{Jonathan M. Fraser, University of St Andrews, Scotland}
\authoremail{jmf32@st-andrews.ac.uk}
\thanks{JMF was financially supported by  a  \emph{Leverhulme Trust Research Project Grant} (RPG-2019-034) and an \emph{EPSRC Standard Grant} (EP/Y029550/1).}

\date{}

\begin{document}


\maketitle
\thispagestyle{empty}

\begin{abstract}
There is a beautiful and well-studied relationship between the Poincar\'e exponent and the fractal dimensions of the limit set of a Kleinian group.  Motivated by this, given an arbitrary discrete subset of the unit ball we define  a critical exponent   and investigate how it relates to the fractal dimensions of the associated generalised limit set.
\\ \\ 
\emph{Mathematics Subject Classification}: primary: 28A80, 30F40; secondary: 28A78, 37F32.
\\
\emph{Key words and phrases}:   limit set, critical exponent, Hausdorff dimension, box dimension, Kleinian group, Poincar\'e exponent.
\end{abstract}

\tableofcontents

\section{Introduction}

\subsection{Dimension theory of Kleinian limit sets}

Let $\dn = \{z \in \mathbb{R}^n : |z|<1\}$ denote the open unit ball in $\rn$.   Temporarily equip $\dn$ with the \emph{hyperbolic metric} $d_H$ defined by
\[
|ds| = \frac{2|dz|}{1-|z|^2}.
\]
The space $(\dn, d_H)$ is commonly referred to as the \emph{Poincar\'e ball model} of hyperbolic space; see \cite{beardon,maskit}. Write  $\text{Con}^+(n)$ for the group of orientation preserving isometries of $(\dn, d_H)$.  This group can be characterised as the orientation preserving M\"obius transformations on $\rn \cup \{\infty\}$ which setwise stablise $\dn$.  A subgroup $\Gamma \leq \text{Con}^+(n)$ is \emph{Kleinian} if it discrete and the  \emph{limit set} of a Kleinian group $\Gamma$ is  
\[
L(\Gamma) = \overline{\Gamma(0)} \setminus \Gamma(0)
\]
 where the closure is the Euclidean closure. The related \emph{radial limit set} is defined by
\[
L_{rad}(E) = \{ z \in S^{n-1}: \exists c\geq 1 \text{ s.t. } \forall r>0, \,  \exists x \in \Gamma(0)  \text{ s.t. } |z-x| \leq c(1-|x|) \leq cr\},
\]
that is, the radial limit set consists of points in the limit set which can be approached from inside a fixed cone.    Kleinian limit sets are often  beautiful fractal objects and a key question is to determine their fractal dimensions.  Central to this theory is the \emph{Poincar\'e exponent} which  provides a coarse  rate of accumulation of the orbit $\Gamma(0)$ to the limit set.  First define the  \emph{Poincar\'e series} $P_\Gamma: [0, \infty) \mapsto [0,\infty]$ by
\[
P_\Gamma(s) = \sum_{g \in \Gamma } \exp(-sd_H(0,g(0))) = \sum_{g \in \Gamma} \left(\frac{1-|g(0)|}{1+|g(0)|} \right)^s
\]
and then the \emph{Poincar\'e exponent} by
\[
\delta'(\Gamma) = \inf\{ s \geq 0 : P_\Gamma(s) <\infty\}.
\]
If $\Gamma$ is non-elementary and geometrically finite (see \cite{bowditch} for the definitions), then seminal work of Patterson \cite{patterson}, Sullivan \cite{sullivan}, Stratmann--Urba\'nski \cite{su} and Bishop--Jones \cite{bishopjones} established that the box and Hausdorff dimensions of $L(\Gamma)$ are given by $\delta'(\Gamma)$. Moreover, a further deep result of Bishop--Jones  \cite{bishopjones} proved that the Hausdorff dimension of the radial limit set is always  $\delta'(\Gamma)$ in the non-elementary case (that is, without assuming geometric finiteness). See the survey \cite{stratmann} for more background and details.

\subsection{Box and Hausdorff dimension and some notation}

We will be concerned with the Hausdorff and box dimensions of $L(E)$ and we briefly  recall the definitions; see \cite{falconer} for more details. Throughout this section we fix   a non-empty, bounded set $F \subseteq \rn$. For convenience we define the Hausdorff and box dimensions of the empty set to be equal to zero. 

We first define the box dimension. We say $\{P_i\}_i$ is a \emph{$\delta$-packing} of $F$ if the sets $P_i$ are pairwise disjoint and  each $P_i$ is a closed ball of radius $\delta >0$ with its centre in $F$.  Denote the largest number of sets in a $\delta$-packing of $F$ by $N_\delta(F)$. The \emph{lower box dimension} and the \emph{upper box dimension} of $F$ are given by
\begin{align*}
\lbd F &= \liminf_{\delta \to 0} \frac{\log N_\delta(F)}{-\log \delta} \\
\ubd F &= \limsup_{\delta \to 0} \frac{\log N_\delta(F)}{-\log \delta}
\end{align*}
respectively. If these agree then we refer to the common value as the \emph{box dimension} of $F$, written $\bd F$. 

Next we define the Hausdorff dimension.  Given  $\delta>0$, we define the  \emph{$\delta$-approximate $s$-dimensional Hausdorff measure} of $F$ by
\[
\cur{H}^s_\delta(F) = \inf \left\{ \sum_{i=1}^\infty |U_i|^s: \text{$\{U_i\}_i$ is a $\delta$-cover of } F \right\}
\]
and  the \emph{$s$-dimensional Hausdorff measure of $F$} by 
\[
\cur{H}^s(F) = \lim_{\delta \to 0} \cur{H}^s_\delta(F).
\]
Using separability of $\rn$, we may assume all $\delta$-covers are countable and if $F$ is compact we can even assume they are finite. Finally, the \emph{Hausdorff dimension} of $F$ is given by
\[
\hd F = \inf\{s\geq 0: \cur{H}^s(F)=0\} = \sup\{s\geq 0: \cur{H}^s(F)=\infty\}.
\]
It is straightforward to show that
\[
0 \leq \hd F \leq \lbd F \leq \ubd F \leq n
\]
always holds, although if $F$ satisfies some additional homogeneity or regularity conditions, then the box and Hausdorff dimensions may agree.

Throughout the  paper, we write $A \lesssim B$ to mean there exists a constant $c >0$ such that $A \leq cB$.  The implicit constants $c$ are suppressed to improve exposition.  If we wish to emphasise that these constants depend on another parameter $\lambda$, then we will write $A \lesssim_\lambda B$.  We also write $A \gtrsim B$ if $B \lesssim A$ and $A \approx B$ if $A \lesssim B$ and $A \gtrsim B$.

We assume throughout that balls are closed unless stated otherwise.  We will mostly work with Euclidean balls, but sometimes we will use hyperbolic balls.  The distinction will be made clear at the time.  Throughout the paper,  $\dn = \{z \in \mathbb{R}^n : |z|<1\}$ is the open unit ball in $\rn$ and $S^{n-1} = \{z \in \mathbb{R}^n : |z|=1\}$ is the $(n-1)$-sphere which is the boundary of $\dn$.

\section{Generalised  limit sets }
 
\subsection{Generalised  limit sets and the critical exponent}

The starting point for this paper is the observation that the limit set, the Poincar\'e series, and the Poincar\'e exponent of a Kleinian group $\Gamma$  only depend on the orbit $\Gamma(0)$, which is, for all intents and purposes, simply a discrete subset of $\dn$.  Our motivating question is to what extent can we recover the dimension theory of limit sets defined purely in terms of arbitrary discrete subsets of $\dn$?

Recall that a set $E \subseteq \rn$ is \emph{discrete} if for all $x \in E$, there exists $r>0$ such that $B(x,r) \cap E = \{x\}$.

\begin{defn} \label{limDef}
Let $E \subseteq \dn$ be  discrete. Then  the \emph{limit set} of $E$ is defined to be
\[
L(E) = \overline{E} \setminus E
\]
where $\overline{E}$ denotes the Euclidean closure of $E$.
\end{defn}

We also define the radial limit set.

\begin{defn} \label{radlimDef}
For a discrete subset $E \subseteq \dn$ and $c \geq 1$, we define the set of \emph{$c$-radial limit points} $L_c(E)$ of $E$ by
\[
L_c(E) = \{z \in S^{n-1}: \forall r>0, \ \exists x \in E   \text{ s.t. } |x-z| \leq c(1-|x|) \leq cr\}
\]
and the  \emph{radial limit set} $L_{rad}(E)$ of $E$ by
\[
L_{rad}(E) = \bigcup_{c \geq 1} L_c(E).
\]    
\end{defn}

It is straightforward to show that $L(E)$ is closed and contained in $\dn \cup S^{n-1}$.  In many cases of interest, $L(E)$ is a fractal set and our main aim is to derive formulae for (or at least estimate) the fractal dimensions of $L(E)$.  Motivated by the Poincar\'e exponent from the theory of Kleinian groups, we associate a critical exponent to the discrete set $E$ which gives a coarse description of how fast it approaches the boundary.

\begin{defn}
Let $E \subseteq \dn$ be discrete. The  associated \emph{accumulation series} $S_E: [0,\infty) \mapsto [0,\infty]$ is given by
\[
S_E(s) = \sum_{x \in E} (1-|x|)^s
\]
where $1-|x|$ is the distance of $x$ to the boundary $S^{n-1}$. 
\end{defn}

Since $1-|x|<1$ for all $x \in E$,   $S_E(s)$ is non-increasing and so has a well-defined critical exponent.  

\begin{defn}
For $E \steq \dn$   discrete, we define the \emph{critical exponent $\delta(E)$} to be the smallest $s$ such that $S_E(s)$ converges, i.e.
\[
\delta(E) = \inf\{ s \geq 0 : S_E(s) <\infty\} = \sup \{ s \geq 0 : S_E(s) =\infty\}
\]
where $\inf \emptyset = \infty$ and $\sup \emptyset = 0$ by convention.
\end{defn}

\subsection{First observations and simple examples}

A naive guess, inspired by the theory of Kleinian limit sets, could be that the dimensions of $L(E)$ are given by the critical exponent $\delta(E)$.  This is obviously completely false in this level of generality.  Moreover, this naive guess can fail as dramatically as possible in either direction.

\begin{ex} \label{example1}
 For each $k \in \mathbb{N}$ let $E_k \subseteq \dn$ consist of $k$ maximally separated  points all of distance $2^{-k}$ from the boundary  $S^{n-1}$ and let $E = \cup_k E_k$.  Then clearly $E$ is discrete, $L(E) =  S^{n-1}$, and so $\hd L(E) = \bd L(E) = n-1$ is as large as possible.  However,  
\[
S_E(s)  = \sum_k k 2^{-ks} <\infty
\]
for all $s>0$ and therefore $\delta(E) = 0$.
\end{ex}

\begin{ex}\label{example2}
Let $E  \subseteq \dn$ consist of a sequence of points $x_k$ converging along a common line to a single point $w \in S^{n-1}$ and suppose $x_k$ is a distance $1/\log k$ from the boundary.  Then clearly $E$ is discrete, $L(E) =  \{w\}$, and so $\hd L(E) = \bd L(E) = 0$ is as small as possible.  However,  
\[
S_E(s)  = \sum_k  (\log k)^{-s} =\infty
\]
for all $s\geq0$ and therefore $\delta(E) = \infty$.
\end{ex}

Another problematic  situation is when the limit set is not a subset of the boundary.  

\begin{ex}\label{example3}
Let $E \steq \dn$ be discrete. If $L(E) \cap \dn \neq \emptyset$, then $\delta(E)=\infty$.    To see this, let $w \in L(E) \cap \dn$.  Then there must be infinitely many distinct elements of $E$ whose distance to the boundary is at least $(1-|w|)/2$.  Therefore $S_E(s)  = \infty$ for all $s\geq0$ and  $\delta(E) = \infty$.
 \end{ex}

With these extreme examples in mind, it is clear that we need to impose some additional structure on $E$ in order to say anything sensible.  Our aim is therefore to make as mild assumptions as possible, and still be able to say something.  \\

\begin{defn} \label{sepWAdef}
Let $E \steq \dn$ be discrete.     

\begin{enumerate}
\item[(i)] We say that $E$ is \emph{separated} if there exists a constant $c_1 \in (0,1)$ such that for all $x \in E$,
\[
B(x, c_1(1-|x|)) \cap E = \{x\}.
\]

\item[(ii)] We say that $E$ is \emph{well-approximated} if there exists some constant $c_2 \geq 1$ such that for all $x \in E$, there is an element $z \in L(E)$ satisfying 
\[
|x-z| \leq c_2(1-|x|).
\]
\end{enumerate}
\end{defn}

It is easy to see that the example described in Example \ref{example1} is not well-approximated (but is separated) and the example presented in Example \ref{example2} is not separated (but is well-approximated).  The examples discussed in Example \ref{example3} are not separated.  That is, if $E \subseteq \dn$ is separated, then $L(E) \cap \dn = \emptyset$. 

We saw in Example \ref{example2} that $\delta(E) = \infty$ is possible, but it is straightforward to see that if $E$ is separated, then $\delta(E) \leq n-1$.  Indeed, partitioning $E$ into  elements $x$ for which $(1-|x|) \approx 2^{-k}$, separated $E$ must satisfy
\[
S_E(s)  \lesssim \sum_k 2^{k(n-1)} 2^{-ks} < \infty
\] 
for $s >n-1$.  However, being separated alone does not allow us to say anything useful about the dimensions of the limit set.

\begin{ex}\label{example4}
 For each $k \in \mathbb{N}$ let $E_k \subseteq \dn$ consist of  a maximal collection of $2^{-k}$-separated   points all of distance $2^{-k}$ from the boundary and distance at most $2^{-\sqrt{k}}$ from a given point $w \in S^{n-1}$  and let $E = \cup_k E_k$.  Then clearly $E$ is discrete, separated, and $L(E) =  \{w\}$, and so $\hd L(E) = \bd L(E) = 0$ is as small as possible.  However,  
\[
S_E(s)  = \sum_k (\# E_k)2^{-ks}  \approx \sum_k \left( \frac{2^{-\sqrt{k}}}{2^{-k}} \right)^{n-1} 2^{-ks} <\infty
\]
if and only if $s\geq n-1$ and therefore $\delta(E) = n-1$.
 \end{ex}

\section{Main results}

\subsection{Dimension bounds}

Our first result establishes a general bound for the upper box dimension of the limit set, under the assumption that the defining set is both separated and well-approximated; see Definition  \ref{sepWAdef}.  It follows from  Examples \ref{example2} and \ref{example4} that neither condition is sufficient on its own to establish the same conclusion.

\begin{thm} \label{sepWAlma}
If $E \subseteq \dn$ is separated and well-approximated, then
\[
\delta(E) \leq \ubd L(E).
\]
\end{thm}

\begin{proof}
The general proof strategy is  to  build  an $r$-packing of the limit set from the set of points in $E$ that are roughly distance $r$ away from the boundary and use this packing to relate the upper box dimension of the limit set and the critical exponent.

We may assume that $E$ is infinite and then $\emptyset \neq L(E) \subseteq S^{n-1}$. For a given $r>0$, let
\[
E_r = \{x \in E: \emph{ } 1-|x| \in [r,2r)\}.
\]
Since $E$ is separated, it follows that $\# E_r < \infty$ and that $E_r$ is non-empty for a sequence of $r \searrow 0$.  Define also
\begin{align*}
Z_r &= \{z \in L(E): \emph{ } \exists x \in E_r \emph{ }s.t. \emph{ }|z-x| \leq c_2(1-|x|)\}    \\
&= L(E) \cap \bigcup_{x \in E_r } B(x, c_2(1-|x|))
\end{align*}
where $c_2 \geq 1$ is as in Definition \ref{sepWAdef} (ii). Since $E$ is well-approximated,   $Z_r$ is non-empty whenever $E_r$ is non-empty. Note that
\[
Z_r \subseteq \bigcup_{x \in E_r} B(x, c_2(1-|x|)) \subseteq \bigcup_{x \in E_r} B(x, 2c_2 r)
\]
and that $Z_r \cap B(x,2c_2 r)$ is non-empty for all $x \in E_r$.  Now, for every $x \in E_r$, choose  exactly one $z \in Z_r \cap B(x,2c_2 r)$ and denote this set of $z$ by $\Tilde{Z}$. Observe that each $z \in Z_r$  is contained in $\lesssim 1$  distinct balls from the collection $ \{B(x, 2c_2 r)\}_{x\in E_r}$. To see this, note that if $z \in B(x,2c_2r)$ then 
\[
B(x, c_1 r/2) \subseteq B(x,2c_2r) \subseteq B(z,4c_2r)
\]
and the balls $\{B(x, c_1 r/2)\}_{x\in E_r}$ are pairwise disjoint since $E$ is separated. Here   $c_1 \in (0,1)$ is as in Definition \ref{sepWAdef} (i). Therefore, writing $vol$ to denote   $n$-dimensional volume,
\begin{align*}
\# \{ x \in E_r : z \in  B(x,2c_2r)\}  \leq    \frac{vol(B(z, 4c_2 r))}{\inf_x vol(B(x, c_1 r/2))} \approx \frac{(4c_2r)^n}{(c_1 r/2)^n} =  8^nc_2/c_1 \approx 1.
\end{align*}
In particular, $ \# \Tilde{Z} \gtrsim \# E_r$.

  We wish to extract an $r$-packing of the limit set from the collection $\{B(z,r)\}_{z \in \Tilde{Z}}$.  Suppose    $z, z' \in \Tilde{Z}$ are distinct with $z \in B(x,2c_2r)$ and $z' \in B(x',2c_2r)$ for distinct $x,x' \in E_r$, and are such that $|z-z'| \leq 2r$. Then $z' \in B(x',2c_2r) \cap B(x, (2c_2+2)r)$, from which we deduce $x' \in B(x, (4c_2+2)r)$. However,  since $E$ is separated, the number of distinct $x'$ contained in the ball $B(x, (4c_2+2)r)$ is at most
\[
\frac{vol(B(x, (4c_2+2)r))}{\inf_{x'} vol(B(x',c_1r/2))} \approx \frac{(4c_2+2)^nr^n}{(c_1r/2)^n}  = \frac{2^n(4c_2+2)^n}{c_1^n}   \lesssim 1.
\]
Therefore,  we can extract a subset $Z \subseteq \Tilde{Z}$ with $\# Z \gtrsim \# \Tilde{Z}$, and where the collection $P_r =  \{B(z,r)\}_{z \in Z}$ is  an $r$-packing of $L(E)$. Therefore, for every non-empty $E_r$ we have built an $r$-packing $P_r$ of $L(E)$ with 
\begin{equation} \label{sizeofpacking}
\# P_r \gtrsim \# E_r.
\end{equation}
Write $\delta = \delta(E)$,  assume without loss of generality that $\delta>0$, and let $\epsilon \in (0, \delta)$. By the definition of $\delta$, 
\begin{align*}
\infty  = S_E(\delta-\epsilon)  \leq 1+  \sum_{k \in \mathbb{N}} \sum_{x \in E_{2^{-k}}} (1-|x|)^{\delta-\epsilon} \lesssim \sum_{k \in \mathbb{N}} \# E_{2^{-k}} \cdot 2^{-k(\delta-\epsilon)} .
\end{align*}
This ensures that  infinitely many $k\in \mathbb{N}$ satisfy 
\begin{equation} \label{sizeofer}
\# E_{2^{-k}} \geq k^{-2} \cdot 2^{k(\delta-\epsilon)}.
\end{equation}
Combining \eqref{sizeofpacking} and \eqref{sizeofer}, we get that for a sequence of $r \searrow 0$ 
\[
N_r(L(E)) \gtrsim  \frac{r^{-(\delta-\eps)}}{(\log(1/r))^2}
\] 
and, since our choice of $\epsilon>0$ was arbitrary,   we conclude $\ubd{L(E)} \geq \delta$, as required.
\end{proof}

Given the previous theorem, it is natural to ask for conditions under which the critical exponent gives an \emph{upper} bound for the dimensions of the limit set.  In fact we will not make any additional assumptions here, but we will only bound the dimension of the \emph{radial} limit set. We demonstrated in Example \ref{example1} that the same bound does not hold in general if we replace the radial limit set with the full limit set.

\begin{thm} \label{radlimthm}
If $E \subseteq \dn$ is  discrete, then 
\[
\hd L_{rad}(E) \leq \delta(E).
\]
\end{thm}
\begin{proof}
For a given $r>0$, let $E^r = \{x \in E: 1-|x| \leq r\}$ and let $c \geq 1$. Note that $\{B(x, c(1-|x|))\}_{x \in E^r}$ is a $2cr$-cover of $L_c(E)$.  Write  $\delta = \delta(E)$, fix $\epsilon>0$ and let $r>0$.  Without loss of generality we may assume $\delta<\infty$.  Then, by the definition of Hausdorff measure,
\begin{align*}
\mathcal{H}_{2cr}^{\delta+\epsilon}(L_c(E)) & \leq \sum_{x \in E^r} (2c(1-|x|))^{\delta+\epsilon} \\
& \leq (2c)^{\delta+\epsilon}r^{\epsilon/2} \sum_{x \in E} (1-|x|)^{\delta+{\epsilon/2}} \\
& = (2c)^{\delta+\epsilon}r^{\epsilon/2} S_E(\delta+{\epsilon/2})
\end{align*}
noting that $ S_E(\delta+{\epsilon/2})<\infty$ by definition.  Letting $r\rightarrow{0}$ we get $\mathcal{H}^{\delta+\epsilon}(L_c(E)) = 0$. Since $\epsilon>0$ was  arbitrary, we conclude that $\hd L_c(E) \leq \delta$. Moreover,
\[
L_{rad}(E) = \bigcup_{c \geq 1} L_c(E) = \bigcup_{\substack{c \geq 1\\
                                                    c \in \mathbb{N}}}
                                                L_c(E)
\]
and so,  by the countable stability of Hausdorff dimension, $\hd L_{rad}(E)  \leq \delta$, as required.  
\end{proof}

Note that in many cases Theorem \ref{radlimthm} can be `upgraded' to give an upper bound for the Hausdorff dimension of the full limit set $L(E)$.  For example, one might know \emph{a priori} that the complement of the radial limit set in the full limit set is at most countable (and therefore the radial limit set carries the full Hausdorff dimension).  Indeed, this is the case for limit sets of geometrically finite Kleinian groups where non-radial limit points are precisely the (countable set of) parabolic fixed points; see \cite{bowditch}.

   Combining Theorems \ref{sepWAlma} and \ref{radlimthm} we get the following result, which gives a precise formula for the Hausdorff and box dimensions under an additional regularity assumption.

\begin{cor} \label{bigThm}
Suppose  $E \subseteq \dn$ is  separated and well-approximated and satisfies  $ \hd L_{rad}(E) = \ubd L(E)$.  Then
\[
\hd L(E) = \dim_{\textup{B}} L(E) = \delta(E).
\]        
\end{cor}

\subsection{Sharpness and characterisation of dimension}

Theorem  \ref{sepWAlma} established that, in the separated and well-approximated case, the critical exponent provides a lower bound for the upper box dimension of the limit set.  In the next theorem we show that this bound is sharp.  Moreover, by viewing this theorem in reverse, that is, viewing an arbitrary closed subset of the boundary as a limit set  in our context, this provides a characterisation of the upper box dimension of an arbitrary compact set.

\begin{thm} \label{boxdimsharpness}
Suppose $X \subseteq S^{n-1}$ is non-empty and closed. Then 
\[
\ubd X = \max_{E} \{\delta(E): L(E)=X\}
\]
where the maximum is taken over all $E \subseteq \dn$ which are separated and well-approximated.
\end{thm}
\begin{proof}
The inequality $\ubd X  \geq \sup_{E } \delta(E )$ follows from Theorem  \ref{sepWAlma}, and we show the reverse (as well as the fact that the maximum exists)  by constructing a set $E $ satisfying the desired properties and  with $\ubd X  \leq \delta(E )$. 

Given $k \in \mathbb{N}$   let $\{y^k_i\}_i$ be a maximal  $2^{-k}$-separated subset of  $S^{n-1}$ and define the sets
\[
E_k = \left\{(1-2^{-k}) \cdot y_i^k :   \inf_{z \in X} |z - y_i^k| \leq 2^{-k}  \right\} \subseteq \dn
\]
and $E = \cup_k E_k$.  Then $L(E) = X$ and $E$ is both separated and well-approximated.  The set $E$ could have been extracted implicitly from the theory of Whitney decompositions (see the Whitney covering lemma) but we prefer to give an explicit and self-contained proof.  Observe that, for each $k \in \mathbb{N}$,
\[
\{ B(x, 2^{1-k})\}_{x \in E_k}
\]
is a cover of $X$ and so 
\[
\sup_k   (\# E_k )  2^{-ks} = \infty
\]
for $s < \ubd X$.  For such $s$,
\[
S_E(s) = \sum_k \sum_{x \in E_k} (1-|x|)^{s} = \sum_k  (\# E_k)   2^{-ks}  = \infty
\]
and so $\delta(E) \geq s $ and we conclude $\ubd X  \leq \delta(E )$ as required.
\end{proof}

Next we consider sharpness for Theorem \ref{radlimthm}. Analogous  to above, we provide a characterisation of the Hausdorff  dimension of an arbitrary compact set.

\begin{thm} \label{radlimNOTInf}
Suppose $X \subseteq S^{n-1}$ is non-empty and closed. Then 
\[
\hd X = \min_{E} \{\delta(E): \radlim(E)=X\}
\]
where the minimum  is taken over all $E \subseteq \dn$ which are separated and well-approximated.
\end{thm}
\begin{proof}
The inequality $\hd X  \leq \inf_{E } \delta(E )$ follows from Theorem  \ref{radlimthm}, and we show the reverse (as well as the fact that the minimum exists)  by constructing a set $E $ with the desired properties and with $\hd X  \geq \delta(E )$. 

Fix $s > \hd X$. Then for all $k \in \mathbb{N}$, there is an index set $I_k$ and a collection of balls $\{B(z_{i_k},r_{i_k})\}_{i_k \in I_k}$ with each  $z_{i_k} \in X$ and $0<r_{i_k} \leq 2^{-k}$ such that
\begin{equation} \label{usinghausdorffmeasure}
X\subseteq \bigcup_{I_k} B(z_{i_k},r_{i_k}) \ \text{ and } \ \sum_{I_k} r_{i_k}^s < 2^{-k}.
\end{equation}
Since $X$ is compact, we may assume $\# I_k < \infty$ for all $k$ and such that
\begin{equation} \label{layers}
\min_{i_k \in I_k} r_{i_k} \geq 2 \max_{i_{k+1} \in I_{k+1}} r_{i_{k+1}}  
\end{equation}
 for all $k$. By the Vitali covering lemma, for each $k \in \mathbb{N}$, there is a subset $J_k \subseteq I_k$ so that the subcollection $B_k = \{B(z_{j_k},r_{j_k})\}_{j_k \in J_k} \subseteq \{B(z_{i_k},r_{i_k})\}_{i_k \in I_k}$ is disjoint and 
\[
X \subseteq \bigcup_{I_k} B(z_{i_k},r_{i_k}) \subseteq \bigcup_{J_k} B(z_{j_k},3r_{j_k}).
\]
We define $E$ using $B_k$. For each $B(z_{j_k},r_{j_k}) \in B_k$, let $x_{j_k}$ be the point on the line segment between $0$ and $z_{j_k}$ with $|x_{j_k}| = 1-r_{j_k}$. Then take
\[
E = \bigcup_k \bigcup_{j_k \in J_k} \{x_{j_k}\}.
\]
First we check $X = \radlim (E)$. Let $w \in \radlim (E)$, in which case $w \in L_c(E)$ for some $c>0$ by definition. For all $r>0$, there exists $x = x_{j_k} \in  E \cap B(w,c(1-|x|)) \subseteq B(w,cr)$ for some $j_k \in J_k$.  But  $z = z_{j_k} \in X$ with $|x-z| = 1-|x|$ and hence, $|w-z| \leq cr$. Since $c$ is fixed, $r>0$ was arbitrary and $X$ is closed, we conclude $w \in X$.

For the backwards inclusion, again let $r>0$ be given and let $z \in X$. Then there exists $k \in \bb{N}$ such that $r_{j_k} \leq r$ for all $j_k \in J_k$, and where $z \in B(z_{j_k},3r_{j_k})$ for some $j_k \in J_k$. Then $x_{j_k} \in E$ with $|z_{j_k} - x_{j_k}| = 1-|x_{j_k}|$. Hence, $|z - x_{j_k}| \leq 3r_{j_k} + 1-|x_{j_k}| = 4(1-|x_{j_k}|) \leq 4r$ and $z \in \radlim(E)$, which proves the claim. Similarly we see $E$ is well-approximated.

Next we show $E$ is separated. Fix $x_{j_k} \in E$ and consider $x_{l_k} \in E$ for some $l \neq j$. Since $B(z_{j_k},r_{j_k})$ and $B(z_{l_k},r_{l_k})$ are disjoint,
\[
|x_{j_k} - x_{l_k}| \geq\frac{1}{2} |z_{j_k} - z_{l_k}|  \geq  \frac{1}{2} \max\{ r_{j_k}, r_{l_k}\}  \geq \frac{1}{2}  (1-|x_{j_k}|). 
\]
Otherwise, for $x_{j'_l} \in E$ with $k \neq l$,
\[
|x_{j_k} - x_{j'_l}| \geq \frac{1}{2} (1-|x_{j_k}|)
\]
by \eqref{layers}. This proves $E$ is separated. 

Finally, using \eqref{usinghausdorffmeasure}
\begin{align*}
S_{E}(s) = \sum_k \sum_{j_k \in J_k} (1-|x_{j_k}|)^s = \sum_k \sum_{j_k \in J_k} r_{j_k}^s \leq \sum_k 2^{-k} < \infty.
\end{align*}
Therefore,  $\delta(E) \leq s$ and since this holds for all $s > \hd X$, we conclude $\delta(E) \leq \hd X$, as required. 
\end{proof}

Note that, unlike in Theorem \ref{boxdimsharpness}, it was not  necessary to assume $E$ is separated and well-approximated in Theorem \ref{radlimNOTInf}.  However, the result is made stronger by including these conditions since we show we can achieve the minimum within a restricted (and well-behaved) family of sets.  

The full limit set is always closed and so the assumption that $X$ is closed in Theorem \ref{boxdimsharpness} does not impose any restrictions.  However, the radial limit set need not be closed, and so ideally we would remove this assumption in Theorem \ref{radlimNOTInf}.  That said, the radial limit set cannot be an arbitrary set and so we first we consider the set theoretic complexity of the radial limit set in general.

\begin{prop}
For an arbitrary non-empty discrete set $E \subseteq \dn$, the associated radial limit set $\radlim(E)$ is a  $\mathcal{G}_{\delta\sigma}$ set.  
\end{prop}

\begin{proof}
Given  $r>0$, define 
\[
E^r = \{ x \in E : 1-|x| <r\}
\]
and, given  $c \geq 1$ and $x \in\dn$, define 
\[
C_{c,x} = \{ z \in S^{n-1}  : |x-z| < c(1-|x|)  \} .
\]
Then, 
\[
\radlim(E) = \bigcup_{c \in \mathbb{N}} \bigcap_{r \in (0,1) \cap \mathbb{Q}} \bigcup_{x \in E^r}C_{c,x}.
\]
Then, noting that $E$ (and so each $E^r$) is countable since it is a discrete subset of $\dn$ and using that $C_{c,x}$ is open, we get that  $\radlim(E)$ is a  $\mathcal{G}_{\delta\sigma}$ set as required.
\end{proof}

Unfortunately, we do not know how to prove Theorem \ref{radlimNOTInf} for arbitrary $\mathcal{G}_{\delta\sigma}$ sets, but we can at least partially extend it to $\mathcal{F}_\sigma$ sets; one class lower in the Borel hierarchy.

\begin{thm} \label{radlimNOTInf2}
Suppose $X \subseteq S^{n-1}$ is a non-empty $\mathcal{F}_\sigma$ set. Then 
\[
\hd X = \min_{E} \{\delta(E): \radlim(E) \supseteq X\}
\]
where the minimum  is taken over all $E \subseteq \dn$ which are separated and well-approximated.
\end{thm}

\emph{Proof sketch.}
Once again, the inequality $\hd X  \leq \inf_{E } \delta(E )$ follows from Theorem  \ref{radlimthm}, and it remains to show the reverse  by constructing a set $E$ with the desired properties and  with $\hd X  \geq \delta(E )$. 

Write $X = \cup_{m} X_m$ where each $X_m$ is non-empty and closed.  By Theorem \ref{radlimNOTInf}, for each $m$ there exists $E_m$ such that 
\[
\hd X_m = \delta(E_m),
\]
 $\radlim(E_m) = X_m$, and $E_m$ is separated and well-approximated.  Setting $E =  \cup_{m} E_m$, we immediately get
\[
\hd X = \sup_m \hd X_m = \sup_m \delta(E_m) \leq \delta(E)
\]
and $\radlim(E) \supseteq \cup_m \radlim(E_m)  = \cup_m X_m =    X$ and also that $E$ is well-approximated.  These observations are not enough to prove the theorem and so in order to upgrade our conclusions consider the countable collection of `scale windows' $\{[2^{-k}, 2^{1-k})\}_{k \in 2\mathbb{N}}$ and partition this collection into countably infinite  many countably infinite subcollections.  We assign the subcollections to the indices $m$ and then insist that all the scales used in the construction of $E_m$ belong only to scale windows in the subcollection associated with $m$.  One has to check that the proof of  Theorem \ref{radlimNOTInf} can be modified in this way, but we leave this to the reader.  This ensures that the sets $E_m$ do not `see each other'.  This is  enough to ensure that $E$ is  separated.  Moreover, when running the proof of  Theorem \ref{radlimNOTInf} for $X_m$  replace \eqref{usinghausdorffmeasure} with
\[
X_m\subseteq \bigcup_{I_k} B(z_{i_k},r_{i_k}) \ \text{ and } \ \sum_{I_k} r_{i_k}^s < 2^{-k}2^{-m}
\]
so that, at the end,
\[
S_{E}(s) = \sum_m \sum_k \sum_{j_k \in J_k} (1-|x_{j_k}|)^s = \sum_m\sum_k \sum_{j_k \in J_k} r_{j_k}^s \leq\sum_m \sum_k 2^{-k}2^{-m} < \infty
\]
for all $s>\hd X$. This ensures  that $\delta(E) \leq  \hd X$, completing the proof.
\hfill \qed

In Theorem \ref{radlimNOTInf2} it would be good to replace $\radlim(E) \supseteq X$ with   $\radlim(E) =  X$ but we do not know how to do this.

\section{Kleinian limit sets---revisited}

In this section, we briefly revisit the setting of Kleinian groups.  Part of the motivation here is to understand exactly what properties of the Kleinian group are needed to make statements about the dimension theory of the limit set. Note that, even though the Poincar\'e series and our series $S_E(s)$ are not precisely the same, it is easy to see that $\delta'(\Gamma) = \delta(\Gamma(0))$ where $\delta(\Gamma(0))$ is the critical exponent of the discrete set  $\Gamma(0)$. From now on we write $\delta(\Gamma)$ for both exponents without distinguishing between them.  In this section we use some basic hyperbolic geometry and elementary facts about Kleinian groups; see \cite{beardon, maskit} for more details.

\begin{lma} \label{klnSEP}
 Suppose $\Gamma$ is a Kleinian group and  $z \in \mathbb{D}^n$ is not fixed by any element of $\Gamma$. Then the orbit $\Gamma(z)$ is separated.  
\end{lma}
\begin{proof}
 Let $r>0$ be  such that the collection of balls
\[
\{B_H(g(z),r)\}_{g \in \Gamma}
\]
is pairwise disjoint, where $B_H$ denotes the closed ball in the hyperbolic metric.  To see that such an $r$ exists, observe that, since $\Gamma$ is discrete, the orbit $\Gamma(z)$ is locally finite and so we can choose $r>0$ such that 
\[
B_H(z,2r) \cap \Gamma(z) = \{z\}.
\]
Now suppose  $y \in B_{H}(g_1(z), r) \cap  B_{H}(g_2(z), r) $ for distinct $g_1,g_2 \in \Gamma$. Then 
\[
d_{H}(z, g_1^{-1}g_2(z)) = d_{H}(g_1(z), g_2(z))  \leq d_{H}(g_1(z), y)+d_{H}(y, g_2(z)) \leq 2r
\]
which gives $z \neq g_1^{-1}g_2(z) \in B_{H}(z,2r)$, a contradiction.  We now want to transfer  `separated in the hyperbolic metric' to   `separated in the Euclidean metric', but this follows immediately since the Euclidean diameter of $B_H(g(z),r)= g(B_H(z,r))$ is 
\[
\approx_{z,r} |g'(z)| \approx 1-|g(z)|
\]
for all $g \in \Gamma$.  In particular,  for all $g \in \Gamma$, $B_E(g(z), c(1-|g(z)|)) \cap \Gamma(z) = \{g(z)\}$ for a uniform constant $c$ depending only on $z$ and $r$, where  $B_E$ denotes the closed ball in the Euclidean metric.
\end{proof}

\begin{lma} \label{klnWA}
 Suppose $\Gamma$ is a non-elementary Kleinian group and  $z \in \mathbb{D}^n$. Then the orbit $\Gamma(z)$ is well-approximated.    
\end{lma}
\begin{proof}
Let   $g \in \Gamma$ be given. Since $\Gamma$ is non-elementary,  $\Gamma$ necessarily contains a loxodromic element $h$ with fixed points $w_1,w_2 \in \LG \subseteq \bdry$. Then the element $ghg^{-1} \in \Gamma$ is loxodromic with fixed points $g(w_1)$ and $g(w_2)$; both of which are also  in the limit set.   We show that $g(z)$ is close to at least one of them.  Let $C \subseteq \dn$ be the unique doubly extended hyperbolic geodesic passing through $w_1,w_2$ and let $x \in C$ minimise the distance $d_H(z,x)$.  Then the Euclidean distance between $g(x)$ and $g(z)$ is 
\begin{equation} \label{est11}
\approx_{z,x} |g'(z)| \approx 1-|g(z)|.
\end{equation}  
Since $g$ preserves the hyperbolic metric, $g(C)$ is the doubly extended hyperbolic geodesic passing through $g(w_1),g(w_2)$ and $g(x)$ and, moreover, $g(C)$ is orthogonal to the boundary.  It then follows by a simple geometric argument that the Euclidean distance between $g(x)$ and at least one of $g(w_1),g(w_2)$ is at most 
\begin{equation} \label{est22}
\sqrt{2}( 1-|g(x)|) \approx 1-|g(z)|.
\end{equation}  
By \eqref{est11} and \eqref{est22} and the triangle inequality, $g(z)$ is within distance  $\approx 1-|g(z)|$ of the limit set  and therefore $\Gamma(z)$ is well-approximated, as required.
\end{proof}

Despite the above, it is not  necessarily true that orbits of  elementary Kleinian group are well-approximated and this explains why $\delta'(\Gamma) >\ubd L(\Gamma)$ can hold in this setting.  

\begin{ex}
A Kleinian group $\Gamma$ generated by a single parabolic element with fixed point $w \in S^{n-1}$ is elementary as its limit set contains only the parabolic fixed point. The orbit $\Gamma(z)$ lies in a horosphere  based at $w$ for all $z \in \dn$. For $g(z) \in \Gamma(z)$ sufficiently close to the boundary, 
\[
|w-g(z)| \geqs (1-|g(z)|)^\frac{1}{2}.
\]
Therefore $\Gamma(z)$ is not well-approximated.
\end{ex}

Given Lemmas \ref{klnSEP} and  \ref{klnWA} and our main results, we have therefore presented a self-contained proof of the following.
\begin{cor}
If  $\Gamma$ is a non-elementary Kleinian group, then
\[
\hd L_{rad}(\Gamma) \leq  \delta(\Gamma)  \leq  \ubd L(\Gamma).
\]
\end{cor}

The next result follows immediately from the result of Bishop--Jones which gives  that $\hd L_{rad}(\Gamma) =  \delta(\Gamma)$ \emph{always} holds for non-elementary $\Gamma$; see \cite{bishopjones}.   However, their result is rather deep and our proof is more straightforward.

\begin{cor}
If a non-elementary Kleinian group $\Gamma$   satisfies $\hd L_{rad}(\Gamma) = \ubd \LG$, then
\[
\hd L(\Gamma) = \bd  L(\Gamma) = \delta(\Gamma).
\]
\end{cor}

Using Theorem \ref{boxdimsharpness} we give a characterisation of the upper box dimension of the    limit set of an arbitrary non-elementary Kleinian group.

\begin{cor}
If $\Gamma$ is a non-elementary Kleinian group, then
\[
\ubd L(\Gamma) = \max_{E} \{\delta(E): L(E) = L(\Gamma)\}
\]
where the maximum is taken over all $E \subseteq \dn$ which are separated and well-approximated.
\end{cor}

We note that one can show directly that   $\max_{E} \{\delta(E):L(E) = L(\Gamma)\}$ is the \emph{convex core entropy} $h_c(\Gamma)$  of $\Gamma$ and thus recover   \cite[Theorem 1]{falk}.

\section{Generalisations and further work}

The definitions of separated and well-approximated from Definition  \ref{sepWAdef} were motivated twofold.  First, they are the simplest and weakest  conditions we found which allowed us to directly relate $\delta(E)$ to the dimensions of $L(E)$.  Second, they are satisfied for  orbits of non-elementary Kleinian groups.  However, given that we are considering a very general setup, it is natural to further weaken these conditions and see what can be recovered.  This is the aim of this section.  

\begin{defn}\label{sepWAdef2} 
Let $E \subseteq \dn$ be discrete and let $0<\beta \leq 1 \leq \alpha$.
\begin{enumerate}
\item[(i)] We say that $E$ is \emph{$\alpha$-separated} if there exists a positive constant $c_1>0$ such that for all $x$ in $E$,
\[
B(x, c_1(1-|x|)^\alpha) \cap E = \{x\}.
\]

\item[(ii)]  We say that $E$ is \emph{$\beta$-well-approximated} if there exists a constant $c_2 \geq 1$ such that for all $x$ in $E$, there exists an element $z$ in $L(E)$ such that 
\[
|x-z| \leq c_2(1-|x|)^\beta.
\]
\end{enumerate}    
\end{defn}

If $E$ is 1-separated, then it is simply separated as in Definition  \ref{sepWAdef}  and if $E$ is 1-well-approximated then it is simply well-approximated as in Definition  \ref{sepWAdef}.  The idea is that as $\beta$ and $\alpha$ move away from 1 we are imposing weaker assumptions on $E$---and  should expect poorer dimension estimates.

\begin{thm} \label{sepWAlma2}
If $E \subseteq \dn$ is $\alpha$-separated and $\beta$-well-approximated for some  $0<\beta \leq 1 \leq \alpha$, then
\[
 \ubd L(E) \geq \delta(E) -2 \big( n\alpha-(n-1)\beta -1\big).
\]
\end{thm}

\begin{proof}
The general proof strategy is the same as the proof of Theorem \ref{sepWAlma}, but there are a few places where extra observations are needed. We may assume that $E$ is infinite and then $\emptyset \neq L(E) \subseteq S^{n-1}$. For a given $r>0$, let
\[
E_r = \{x \in E: \emph{ } 1-|x| \in [r,2r)\}.
\]
Since $E$ is $\alpha$-separated, it follows that $\# E_r < \infty$ and that $E_r$ is non-empty for a sequence of $r \searrow 0$.  Define also
\begin{align*}
Z^\beta_r &= \{z \in L(E): \emph{ } \exists x \in E_r \emph{ }s.t. \emph{ }|z-x| \leq c_2(1-|x|)^\beta\}    \\
&= L(E) \cap \bigcup_{x \in E_r } B(x, c_2(1-|x|)^\beta)
\end{align*}
where $c_2 \geq 1$ is as in Definition \ref{sepWAdef2} (ii). Since $E$ is $\alpha$-well-approximated,   $Z^\beta_r$ is non-empty whenever $E_r$ is non-empty. Note that
\[
Z^\beta_r \subseteq \bigcup_{x \in E_r} B(x, c_2(1-|x|)^\beta) \subseteq \bigcup_{x \in E_r} B(x, 2^\beta c_2 r^\beta)
\]
and that $Z^\beta_r \cap B(x,2^\beta c_2 r^\beta)$ is non-empty for all $x \in E_r$.  Now, for every $x \in E_r$, choose  exactly one $z \in Z^\beta_r \cap B(x,2^\beta c_2 r^\beta)$ and denote this set of $z$ by $\Tilde{Z}$. Observe that each $z \in Z^\beta_r$  is contained in 
\[
\lesssim r^{n(\beta-\alpha)-\beta+1}
\]
 distinct balls in the collection $ \{B(x, 2^\beta c_2 r^\beta)\}_{x\in E_r}$. To see this, write
\[
E(r) = E_{r/2} \cup E_{r} \cup E_{2r}
\]
and note that if $z \in B(x,2^\beta c_2r^\beta)$ and $r>0$ is sufficiently small, then 
\[
B(x, c_1 r^\alpha/2) \subseteq E(r) \cap B(x,2^\beta c_2r^\beta) \subseteq  E(r) \cap B(z,2^{1+\beta} c_2r^\beta)
\]
and the balls $\{B(x, c_1 r^\alpha/2)\}_{x\in E_r}$ are pairwise disjoint since $E$ is $\alpha$-separated. Here   $c_1 \in (0,1)$ is as in Definition \ref{sepWAdef} (i). Therefore, writing $vol$ to denote   $n$-dimensional volume,
\begin{align*}
\# \{ x \in E_r : z \in  B(x,2^\beta c_2r^\beta)\}  \leq    \frac{vol\big( E(r) \cap B(z,2^{1+\beta} c_2r^\beta)\big)}{\inf_x vol\big(B(x, c_1 r^\alpha/2)\big)} \approx \frac{r^{(n-1)\beta+1}}{r^{n\alpha}} \approx   r^{(n-1)\beta+1-n\alpha}
\end{align*}
In particular, $ \# \Tilde{Z} \gtrsim r^{n\alpha-(n-1)\beta-1} \cdot \# E_r$.

  We wish to extract an $r$-packing of the limit set from the collection $\{B(z,r)\}_{z \in \Tilde{Z}}$.  Suppose    $z, z' \in \Tilde{Z}$ are distinct with $z \in B(x,2^\beta c_2r^\beta)$ and $z' \in B(x',2^\beta c_2r^\beta)$ for distinct $x,x' \in E_r$, and are such that $|z-z'| \leq 2r$. Then $z' \in B(x',2^\beta c_2r^\beta) \cap B(x, (2^\beta c_2+2)r^\beta)$, from which we deduce $x' \in E(r) \cap B(x, (2^{1+\beta}c_2+2)r^\beta)$ for sufficiently small $r$. However,  since $E$ is $\alpha$-separated, the number of distinct $x' \in E_r$ contained in the ball $B(x, (2^{1+\beta}c_2+2)r^\beta)$ is at most
\[
\frac{vol\big(E(r) \cap B(x, (2^{1+\beta}c_2+2)r^\beta)\big)}{\inf_{x'} vol(B(x',c_1r^\alpha/2))} \approx  r^{(n-1)\beta+1-n\alpha}.
\]
Therefore,  we can extract a subset $Z \subseteq \Tilde{Z}$ with
\[
\# Z \gtrsim r^{n\alpha-(n-1)\beta-1} \cdot \# \Tilde{Z}
\]
and where the collection $P_r =  \{B(z,r)\}_{z \in Z}$ is  an $r$-packing of $L(E)$. Therefore, for every non-empty $E_r$ we have built an $r$-packing $P_r$ of $L(E)$ with 
\begin{equation} \label{sizeofpacking2}
\# P_r \gtrsim r^{2(n\alpha-(n-1)\beta-1)} \cdot \# E_r.
\end{equation}
Write $\delta = \delta(E)$,  assume without loss of generality that $\delta>0$, and let $\epsilon \in (0, \delta)$. By the definition of $\delta$, 
\begin{align*}
\infty  = S_E(\delta-\epsilon)  \leq 1+  \sum_{k \in \mathbb{N}} \sum_{x \in E_{2^{-k}}} (1-|x|)^{\delta-\epsilon} \lesssim \sum_{k \in \mathbb{N}} \# E_{2^{-k}} \cdot 2^{-k(\delta-\epsilon)} .
\end{align*}
This ensures that  infinitely many $k\in \mathbb{N}$ satisfy 
\begin{equation} \label{sizeofer2}
\# E_{2^{-k}} \geq k^{-2} \cdot 2^{k(\delta-\epsilon)}.
\end{equation}
Combining \eqref{sizeofpacking2} and \eqref{sizeofer2}, we get that for a sequence of $r \searrow 0$ 
\[
N_r(L(E)) \gtrsim  r^{2(n\alpha-(n-1)\beta-1)} \frac{r^{-(\delta-\eps)}}{(\log(1/r))^2}
\] 
and, since our choice of $\epsilon>0$ was arbitrary,   we conclude 
\[
\ubd{L(E)} \geq \delta -2(n\alpha-(n-1)\beta-1),
\]
 as required.
\end{proof}

The estimate from Theorem \ref{sepWAlma2}    recovers the estimate from Theorem \ref{sepWAlma} as $\beta,\alpha \to 1$.  Moreover, the estimate is monotonic in the sense that  it improves as $\alpha$ decreases or $\beta$ increases.  If the conditions are weakened too much, then Theorem  \ref{sepWAlma2} becomes trivial, that is, the lower bound drops below zero.   Setting $\beta=1$ we see that in order to get non-trivial information, we need
\[
\alpha<\frac{\delta(E)}{2n}+1.
\]
On the other hand, setting $\alpha=1$ we see that in order to get non-trivial information, we need
\[
\beta > 1-\frac{\delta(E)}{2(n-1)} .
\]
Theorem \ref{sepWAlma2}  is sharp in the following sense.

\begin{ex}
Fix $0<\beta \leq 1 \leq \alpha$.  For each $k \in \mathbb{N}$ let $E_k \subseteq \dn$ consist of  a maximal collection of $2^{-k\alpha}$-separated   points all of distance between $2^{-k}$ and $2^{1-k}$ from the boundary and distance at most $2^{-k\beta}$ from a given point $w \in S^{n-1}$  and let $E = \cup_k E_k$.  Then clearly $L(E) = \{w\}$  and $E$ is $\alpha$-separated and $\beta$-well-approximated.  Moreover,  
\[
S_E(s)  \approx \sum_k (\# E_k)2^{-ks}  \approx \sum_k \frac{2^{-\beta k(n-1)}2^{-k}}{2^{-k\alpha n} } 2^{-ks} <\infty
\]
if and only if $s> n\alpha-(n-1)\beta -1$ and therefore $\delta(E) = n\alpha-(n-1)\beta -1$.  Finally, note the lower bound coming from Theorem \ref{sepWAlma2} is $\ubd L(E) \geq 0$ which is sharp.
\end{ex}

It is also natural to weaken  the definition of the radial limit, for example, to points which can be approached within an appropriate solid of revolution (replacing a cone).

\begin{defn}
Let $E \subseteq \dn$ be discrete and let $\gamma \in (0,1]$. For $c \geq 1$ we define
\[
L_c^\gamma(E) = \{z \in L(E): \forall r>0, \ \exists x \in E \ s.t. \ |x-z| \leq c(1-|x|)^\gamma \leq cr^\gamma\}
\]
and the \emph{$\gamma$-radial limit set} $L_{rad}^\gamma(E)$ of $E$ by
\[
L_{rad}^\gamma(E) = \bigcup_{c \geq 1} L_c^\gamma(E).
\]
\end{defn}

If $\gamma=1$, then $L_{rad}^1(E)$ is simply the usual radial limit set of $E$ and as $\gamma$ decreases the set $L_{rad}^\gamma(E)$ becomes larger and so we should expect a poorer upper bound for the dimension. 

\begin{thm}  \label{radlimthm2}
If $E \subseteq \dn$ is  discrete and $\gamma \in (0,1]$, then 
\[
\hd L_{rad}^\gamma(E) \leq \frac{\delta(E)}{\gamma}.
\]
\end{thm}

\begin{proof}
The general proof strategy is the same as the proof of Theorem \ref{radlimthm}.  For a given $r>0$, let $E^r = \{x \in E: 1-|x| \leq r\}$ and let $c \geq 1$. Note that $\{B(x, c(1-|x|)^\gamma)\}_{x \in E^r}$ is a $2cr^\gamma$-cover of $L_c^\gamma(E)$.  Write  $\delta = \delta(E)$, fix $\epsilon>0$ and let $r>0$.  Without loss of generality we may assume $\delta<\infty$.  Then, by the definition of Hausdorff measure,
\begin{align*}
\mathcal{H}_{2cr^\gamma}^{\delta/\gamma+\epsilon}(L_c^\gamma(E)) & \leq \sum_{x \in E^r} (2c(1-|x|)^\gamma)^{\delta/\gamma+\epsilon} \\
& \leq (2c)^{\delta/\gamma+\epsilon}r^{\epsilon\gamma/2} \sum_{x \in E} (1-|x|)^{\delta+{\epsilon\gamma/2}} \\
& = (2c)^{\delta/\gamma+\epsilon}r^{\epsilon\gamma/2} S_E(\delta+{\epsilon\gamma/2})
\end{align*}
noting that $ S_E(\delta+{\epsilon\gamma/2})<\infty$ by definition.  Letting $r\rightarrow{0}$ we get $\mathcal{H}^{\delta/\gamma+\epsilon}(L_c^\gamma(E)) = 0$. Since $\epsilon>0$ was  arbitrary, we conclude that $\hd L_c^\gamma(E) \leq \delta/\gamma$. Moreover,
\[
L_{rad}^\gamma(E) = \bigcup_{c \geq 1} L_c^\gamma(E) = \bigcup_{\substack{c \geq 1\\
                                                    c \in \mathbb{N}}}
                                                L_c^\gamma(E)
\]
and so,  by the countable stability of Hausdorff dimension, $\hd L_{rad}^\gamma(E)  \leq \delta/\gamma$, as required.  
\end{proof}

Theorem \ref{radlimthm2} is also sharp, as we show below. We use a self-similar construction and refer the reader to \cite{falconer} for more background on self-similar sets.

\begin{ex}
Let $0<t \leq \gamma \leq 1$ and $X \subseteq S^{1}$ be a   self-similar set (in the arclength metric)  generated by 2 contractions with contraction ratios both equal to $2^{-\gamma/t} \in (0,1/2]$ and satisfying the open set condition.  For $k \in \mathbb{N}$ let $E_k \subseteq \mathbb{D}^2$ consist of  points  placed at a distance $2^{-k/t}$ from the boundary close to each of the $2^{k}$ cylinders generating $X$ at the $k$th level.  Let $E = \cup_k E_k$.  Then
\[
S_E(s)  = \sum_k (\# E_k) 2^{-ks/t}  \approx \sum_k    2^{k(1-s/t)}  <\infty
\]
if and only if $s> t$ and therefore $\delta(E) = t$.  Since the cylinders generating $X$ at the $k$th level are of diameter $\approx 2^{-k\gamma/t}$, we get $L_{rad}^\gamma(E) = X$ and so 
\[
\hd L_{rad}^\gamma(E) = \hd X = \frac{t}{\gamma}  = \frac{\delta(E)}{\gamma} 
\]
and the bound from Theorem \ref{radlimthm2} is  sharp.
\end{ex}



\begin{thebibliography}{99}


\bibitem[B83]{beardon}
A. F. Beardon.
{\em The geometry of discrete groups},
Graduate Texts in Mathematics, vol. 91, Springer-Verlag, New York, (1983).





\bibitem[BJ97]{bishopjones}
C. J. Bishop  and P. W. Jones. Hausdorff dimension and Kleinian groups, 
\emph{Acta Math.}, {\bf 179}, (1997), 1--39.

\bibitem[B93]{bowditch}
B. H. Bowditch.  Geometrical finiteness for hyperbolic groups, \emph{J. Funct. Anal.}, {\bf 113}, (1993), 245--317.

\bibitem[F14]{falconer}
K. J. Falconer.
{\em Fractal Geometry: Mathematical Foundations and Applications},
 John Wiley \& Sons, Hoboken, NJ, 3rd. ed., (2014).

\bibitem[FM15]{falk}
K. Falk and K. Matsuzaki.
The critical exponent, the Hausdorff dimension of the limit set and the convex core entropy of a Kleinian group,
\emph{Conform. Geom. Dyn.}, {\bf 19}, (2015), 159--196.



\bibitem[M88]{maskit}
B. Maskit. \emph{Kleinian groups}, Grundlehren der Mathematischen Wissenschaften [Fundamental Principles of Mathematical Sciences], {\bf 287}, Springer-Verlag, Berlin, (1988).


\bibitem[P76]{patterson}
S. J. Patterson.
The limit set of a Fuchsian group,
 {\em  Acta Math.}, {\bf 136}, (1976), 241--273.

\bibitem[S04]{stratmann}
B. O. Stratmann.
The Exponent of Convergence of Kleinian Groups; on a Theorem of Bishop and Jones,
\emph{Progress in Probability}, Birkh\"auser Verlag Basel/Switzerland {\bf 57}, 93--107, (2004).

\bibitem[SU96]{su}
B. O. Stratmann and M.  Urba\'nski.
The box-counting dimension for geometrically finite Kleinian groups, {\em  Fund. Math.}, {\bf 149}, 83--93, (1996). 


\bibitem[S84]{sullivan}
D. Sullivan.
 Entropy, Hausdorff measures old and new, and limit sets of geometrically finite Kleinian groups,
 {\em Acta Math.}, {\bf 153}, (1984), 259--277.




\end{thebibliography}
\end{document}